\documentclass[12pt]{amsart}
\usepackage{amsfonts}
\usepackage{amssymb}
\usepackage{amsmath}
\usepackage{amsthm}
\usepackage{verbatim}
\usepackage{esint}
\usepackage{color}

\newcommand{\dom}{\partial\Omega}

\newcommand{\diff}{\ensuremath{\:\text{d}}}
\DeclareMathOperator*{\supp}{supp}

\newcommand{\HS}{\ensuremath{\text{\it HS}^1}}
\newcommand{\HSt}[1]{\ensuremath{\text{\it HS}_{#1,\text{ato}}^1}}
\newcommand{\tilN}[1]{\ensuremath{N(\nabla #1)}}

\newcommand{\loc}{\ensuremath{\text{loc}}}
\newcommand{\divt}{\ensuremath{\text{div}}}

\normalbaselines 

\newtheorem{theorem}{Theorem}[section]
\newtheorem{lemma}[theorem]{Lemma}
\newtheorem{corollary}[theorem]{Corollary}

\newtheorem{definition}{Definition}[section]

\theoremstyle{definition}

\begin{document}

\title{The regularity problem for elliptic operators with boundary data in Hardy-Sobolev space $\HS$}

\begin{abstract}
Let $\Omega$ be a Lipschitz domain in $\mathbb R^n,n\geq 3,$ and
$L=\divt A\nabla$ be a second order elliptic operator in
divergence form. We will establish that the solvability of the
Dirichlet regularity problem for boundary data in Hardy-Sobolev
space $\HS$ is equivalent to the solvability of the Dirichlet
regularity problem for boundary data in $H^{1,p}$ for some
$1<p<\infty$. This is a \lq\lq dual result" to a theorem in
\cite{DKP09}, where it has been shown that the solvability of the
Dirichlet problem with boundary data in $\text{BMO}$ is equivalent
to the solvability for boundary data in $L^p(\partial\Omega)$ for
some $1<p<\infty$.
\end{abstract}

\author{Martin Dindo\v{s} \and Josef Kirsch}

\maketitle \markright{THE REGULARITY PROBLEM FOR ELLIPTIC
OPERATORS IN $\HS$}

\section{Introduction}
We shall prove an equivalence between solvability of certain
end-point Dirichlet regularity problem in $\HS$ for second order
elliptic operators and the solvability of the Dirichlet regularity
problem with boundary data in $H^{1,p}$ for some $1<p<\infty$. The
space $\HS$ is defined in section 2.

To be more precise, we study the regularity problem for elliptic
operators in divergence form $L=\text{div} A\nabla$ on a Lipschitz
domain $\Omega\subset \mathbb R^n,n\geq 3$. The matrix
$A=(a_{ij}(X))$ has real, bounded measurable coefficients such
that there exists
$\lambda>0$ with $\lambda^{-1}|\xi|^2\leq \sum_{ij}a_{ij}(X)\xi_i\xi_j$ for all $\xi\in\mathbb R^n$ and all $X\in \Omega$.\\

For these elliptic operators the Lax-Milgram Theorem implies that
for every $f\in H^{\frac{1}{2},2}(\partial\Omega)$ there exists a
unique weak solution $u\in H^{1,2}(\Omega)$, i.e.
$$\int_{\Omega} A\nabla u\cdot \nabla\varphi=0\quad\text{for all $\varphi\in C^{\infty}_0(\Omega)$}$$
with $u\equiv f$ on $\partial\Omega$, which means that the
Dirichlet problem
\begin{align*}
Lu &= 0\text{ in }\Omega\\
u&\equiv f\text{ on }\partial \Omega
\end{align*}
is solvable for boundary data in
$H^{\frac{1}{2},2}(\partial\Omega)$. The question, if solvability
still holds for other classes of boundary values, was extensively
studied. In \cite{LSW63} it was shown that the continuous
Dirichlet problem is solvable for these elliptic operators, i.e.
for every $f\in C^0(\partial\Omega)$ there exists a unique $u\in
W^{1,2}_{\loc}(\Omega)
\cap C^0(\bar{\Omega})$ such that $Lu=0$ in $\Omega$ and $u\equiv f$ on $\partial \Omega$.\\

Historically the study of the Dirichlet problem with boundary data
in $L^p$ for elliptic operators of the form $L=\divt A\nabla$ was
initiated by B.E.J. Dahlberg in \cite{Dah77}, where the Laplacian
on Lipschitz domains was considered (the pullback of the Laplacian
on a Lipschitz domain leads to an operator of the form $L=\divt A
\nabla$ for $A$ elliptic with bounded, measurable coefficients).

Apart from the Dirichlet boundary value problem with data in $L^p$
of great interests are also other boundary value problems in
particular the $L^p$ Neumann problem and Dirichlet regularity
problem (or just Regularity problem) where the data are in
$$H^{1,p}(\dom)=\{f\in L^p(\dom);\,\nabla_T f\in L^p(\dom)\}.$$

Our result is motivated by a recent result \cite{DKP09} that
established that the Dirichlet problem with boundary data in
$L^p(\partial \Omega)$ is solvable (abbreviated $(D)_p$) for some
$1< p <\infty$ if and only the Dirichlet problem with boundary
data is solvable in the end-point BMO space (abbreviated
$(D)_{BMO}$).

By the theory of Muckenhaupt's $B_p$-weights it is well known that
$(D)_p$ implies $(D)_q$ for $q\in (p-\varepsilon,\infty)$ and some
$\varepsilon>0$, i.e. solvability is open with respect to $p$ on
$(1,\infty)$. The result in \cite{DKP09} establishes that this
\lq\lq extrapolation property" also holds at the endpoint where
the correct endpoint is $(D)_{BMO}$. Furthermore the $(D)_{BMO}$
solvability is also equivalent to the fact that the harmonic
measure for the operator $L$ is an $A_{\infty}(d\sigma)$ weight
with respect to the surface measure.\vglue1mm

The most classical method for solving these types of boundary
value problems (at least for {\it symmetric} operators with
coefficients of sufficient smoothness) is the method of layer
potentials \cite{FJR} for the Laplacian in ${\mathbb R}^n$ and
\cite{MT1}-\cite{MT3} for variable coefficients operators. What
has been observed are intriguing relationships between various
boundary value problems. Of particular note is the duality between
the $L^p$ Dirichlet boundary value problem and $H^{1,p'}$
Regularity problem ($\frac1p+\frac1{p'}=1$). It turns out that the
the $L^p$ Dirichlet boundary value problem is solvable if and only
if the $H^{1,p'}$ Regularity problem is solvable for the same
operator (assuming {\it symmetry} of the operator). \vglue1mm

We note that our assumptions do not allow to use the method of
layer potentials, but this informal duality  led us to hypothesize
and later prove that the result from \cite{DKP09} does have a
corresponding dual result. We observed that the dual of the Hardy
space is the BMO space and this leads to hypothesis that the
correct endpoint space for the Regularity problem is the atomic
Hardy space. \vglue1mm

Before we formulate our main result precisely we introduce few
necessary definitions. The study of boundary data in $L^p(\partial
\Omega)$ is related to the study of the non-tangential maximal
function, see for example \cite{Dah79}.

\begin{definition}
For $\kappa>1$ we define the cone-like family of non-tangential
approach regions $\{\Gamma_{\kappa}(Q)\}_{Q\in\partial \Omega}$ by
$$\Gamma_{\kappa}(Q)=\{X\in \Omega:|X-Q|<\kappa\:\text{dist}(X,\partial \Omega)\}.$$
We will omit the index $\kappa$ and write $\Gamma(Q)$, if no
confusion can arise. The non-tangential maximal function for the
non-tangential approach region $\{\Gamma(Q)\}_{Q\in\partial
\Omega}$ is defined by
$$u^*(Q)=\sup_{X\in \Gamma} |u(X)|.$$
The truncation at height $h$ of the non-tangential maximal function is defined by $(u)_h^*(Q)=\sup_{X\in\Gamma(Q)\cap B_h(Q)} |u(X)|$.\\
Moreover we define the following variant of the non-tangential
maximal function:
\begin{equation}\label{NTMaxVar}
N(h)(Q)
=\sup_{X\in\Gamma(Q)}\left(\fint_{B_{\frac{\delta(X)}{2}}(X)}|h(Y)|^2\diff
Y\right)^{\frac{1}{2}}\qquad h\in L^2_{loc}(\Omega).
\end{equation}
\end{definition}
\begin{definition}\label{DefDpCondition}
The Dirichlet problem with boundary data in $L^p(\partial
\Omega),1< p <\infty$, is solvable (abbreviated $(D)_p$), if there
exists a constant $C>0$ such that for every $f\in C^0(\partial
\Omega)$ the corresponding unique weak solution $u\in
W^{1,2}_{loc}(\Omega)\cap C^0(\bar{\Omega})$ satisfies
$$||u^*||_{L^p(\partial \Omega)}\leq C ||f||_{L^p(\partial \Omega)}.$$
\end{definition}

\begin{definition}
The regularity problem with boundary data in
$H^{1,p}(\partial\Omega),1<p<\infty,$ is solvable (abbreviated $(R)_{p}$), if
for every $f\in H^{1,p}(\partial\Omega)\cap C^0(\partial \Omega)$
the weak solution $u$ to the problem
\begin{align*}
\begin{cases}
Lu &=0 \quad\text{ in } \Omega\\
u|_{\partial B} &= f \quad\text{ on } \partial \Omega
\end{cases}
\end{align*}
satisfies
\begin{align}
\nonumber \quad||\tilN{u}||_{L^p(\partial \Omega)} +
||u||_{L^p(\Omega)}\leq C||f||_{H^{1,p}(\partial\Omega)}
\end{align}
for a constant $C$ independent of $f$. Similarly, we say that the
regularity problem with boundary data in $\HS(\partial\Omega)$
(abbreviated $(R)_{\HS}$) if for every $f\in
\HS(\partial\Omega)\cap C^0(\partial \Omega)$ the solution $u$
satisfies the estimate
\begin{align}
||\tilN{u}||_{L^1(\partial \Omega)} + ||u||_{L^1(\Omega)}\leq
C||f||_{\HS(\partial\Omega)}.\label{RHSATO}
\end{align}
\end{definition}
We define $(D^*)_{p'}=(D)_{p'}^{L^*}$ for $L^*=\divt A^T \nabla$.
Now we can formulate the main result of this paper:

\begin{theorem} Let $L$ be a divergence form elliptic operator
satisfying the ellipticity condition on a Lipschitz domain
$\Omega$. Then the following two statements hold:
\begin{itemize}

\item If $(R)_{\HS}$ is solvable then $(D^*)_{\text{BMO}}$ and
$(R)_p$ are also solvable for some $1<p<\infty$. Moreover, under
this assumption
$$(R)_p \text{ is solvable if and only if } (D^*)_{p'} \text{ is solvable for }p'=p/(p-1).$$

\item If $(R)_p$ is solvable for some $1<p<\infty$ so is
$(R)_{\HS}$.
\end{itemize}
\end{theorem}

We note that the second part of this statement is not new and
appears in \cite{KP93} (at least for symmetric operators). The
reverse direction is new. Simultaneously, the first part of this
statement improves the result of Shen \cite{She07} (again only
stated for symmetric operators). Shen has established that the
statement
$$(R)_p \text{ is solvable if and only if } (D^*)_{p'} \text{ is solvable for }p'=p/(p-1),$$
holds provided $(R)_q$  is solvable for at least one $q\in
(1,\infty)$. In our statement this is replaced by the $(R)_{\HS}$
is solvability.\vglue1mm

\paragraph*{\bf Acknowledgments.} The second author thanks his Ph.D supervisor Martin Dindo\v{s} for the encouragement and for many helpful discussions and suggestions.
The second author is also grateful to Jill Pipher for her help
concerning the endpoint of Gehring's lemma. The results in this
paper also appear in the Ph.D thesis \cite{Kir11}, where amongst
other things the same problem for elliptic operators with small
drift terms is considered.

\section{Lipschitz Domains and the Hardy-Sobolev Space $\HS$}

In this section we will follow \cite{BD09} to introduce the
Hardy-Sobolev space $\HS$ on the boundary of a Lipschitz domain.
\begin{definition}\label{DefLipDomain}
A domain $\Omega\subset\mathbb R^n$ is called a Lipschitz domain,
if there exist a finite sequence $\{Q_k\}_k\in \partial\Omega$ and
$R_0>0$ such that
\begin{itemize}
\item $\Omega\cap B_{8R_0}(Q_k)$ is in some local coordinates
$\{(x,\phi_k(x)+t):x\in\mathbb R^{n-1},t>0\}\cap B_{8R_0}(0)$ for
a Lipschitz function $\phi_k$ \item $\partial\Omega = \bigcup_k
B_{R_0}(Q_k)\cap\partial\Omega$
\end{itemize}
\end{definition}
Throughout the whole paper, we will assume that $\Omega$ is a
bounded Lipschitz domain in $\mathbb R^n$ for $n\geq 3$. By
definition $\Omega$ is locally the area above a Lipschitz graph
$\varphi$ and so for $Q=(x',\varphi(x'))\in\partial\Omega$ we
define $A_R(Q)=(x',\varphi(x')+R)$ and for $X\in \Omega$ we define
$\hat{X}\in\partial\Omega$ such that $A_R(\hat{X})=X$ for an
appropriate $R$. Thus $A_R(Q)$ and $\hat{X}$ are well defined in each $\partial\Omega\cap B_{8R_0}(Q_k)$. This means that $A_R(Q)$ and $\hat{X}$ depend on $k$, but we will omit the index $k$ to maintain an easy readable notation. If we speak about an $A_R(Q)$ for $R>R_0$ we mean an appropriate point (which will be clear by the context) in $\Omega$, which has distance to $\partial\Omega$ comparable to $1$. The radius of a ball $B$ is denoted by $r(B)$ and
for $Q\in\partial\Omega,X\in \Omega$ and $R>0$ we write:
\begin{align*}
\Delta_R(Q) &= \partial \Omega\cap B_R(Q),\,\,\,\qquad T_R(Q) = \Omega\cap B_R(Q),\\
\delta(X) &=\text{dist}(X,\partial \Omega),\quad\qquad(\partial \Omega)_{\beta} =\{X\in\Omega: \delta(X)<\beta\},\\
\Omega_{\beta} &= \Omega\backslash (\partial \Omega)_{\beta}.
\end{align*}
\vskip0.2cm In \cite{Miy90} and \cite{Cal72} it was shown that a
function having weak derivatives in the Hardy space $H^p$ is
equivalent to a maximal function used by A. P. Calder\'{o}n and
then by A. Miyachi being bounded on $L^p$. In \cite{DS84}, Theorem
5.3, R. Devore and C. Sharpley showed that the maximal function
defined by A. P. Calder\'{o}n is equivalent to a maximal function,
which we will define now for the case regarding one derivative
(see \cite{DS84} (2.2), (4.3), Lemma 2.1, page 36 and page 104 and
\cite{BD09}):
\begin{definition}[]\label{DefHSMaxFct}
Let $\Gamma$ be a domain in $\mathbb R^n$. For $0<q\leq 1$ and
$f\in L^q_{loc}(\Gamma)$ we define the maximal function $f^b_q$ by
$$f^b_q(x)=\sup_{B\ni x}\inf_{c\in{\mathbb R}^n} \frac{1}{r(B)}\left(\fint_B|f-c|^q\right)^{\frac{1}{q}},$$
where the supremum is taken over all balls $B$, which are
contained in $\Gamma$ and contain $x$. The space $\mathcal C^q$ is
defined as all $f\in L^q_{\loc}(\Gamma)$ such that the norm
$$||f||_{\mathcal C^q}=||f_q^b||_{L^q(\Gamma)} + ||f||_{L^q(\Gamma)}$$
is finite.
\end{definition}
For $q=1$ we see that $f_1^b(x)=\sup_{B\ni
x}\frac{1}{|B|^{\frac{1}{n}}} \fint_B |f-f_B|$ with $f_B=\fint_B f
= \frac{1}{|B|}\int_B f$, whereas for $q<1$ the function $f$ might
not be locally integrable and so $f_B$ might not be defined. To
simplify the notation we will write $Nf=f_1^b$, keeping the same
notation in \cite{BD09}.

In \cite{Haj03} (see (6) in \cite{BD09} as well) it was proved for
$f\in \mathcal C^1$, $\frac{s}{s+1}\leq q<1$, where $s$ is a
constant larger than $2$, which depends on the doubling property
of the underlying metric space, and $q^*=\frac{sq}{s-q}$ that
\begin{align}\label{6inBD09}
\left(\fint_{B_r}|f-f_{B_r}|^{q^*}\right)^{\frac{1}{q^*}}\leq C
r\left(\fint_{\lambda B_r} |Nf|^q\right)^{\frac{1}{q}}
\end{align}
for some $\lambda>1$, which is independent of $f$ and $r$. We
define
$$\mathcal M_qf(x)=\sup_{B\ni x}\left(\fint_B |f|^q\right)^{\frac{1}{q}},$$
where the supremum is taken over all balls containing $x$.

In \cite{BD09} N. Badr and G. Dafni proved a relationship between
the Hardy-Sobolev space and the space $\mathcal C^1$ on complete
Euclidian manifolds $M$ with $\mu(M)=\infty$ and $\mu$ a doubling
measure. Since we would like to apply this result later on to
boundary data on $\partial \Omega$ for $\Omega$ a Lipschitz
domain, we will not work in such a general setting. Our domain
will be $\partial\Omega$ for $\Omega$ a Lipschitz domain, where
the surface measure is the underlying measure. Therefore our
domain is bounded and has a finite doubling measure. We will not
write $\partial \Omega$, if there is no confusion possible, which
domain is meant. Similar to Definition 2.11 and Definition 4.3 in
\cite{BD09} and \cite{BB10} we define
\begin{definition}\label{DefinitionAtom}
For $1<t\leq \infty$ we say that a function $a$ is a Hardy-Sobolev
$(1,t)$-atom, if
\begin{itemize}
\item $a$ is supported in a ball $B$ \item $||a||_{L^t}+||\nabla
a||_{L^t}\leq \frac{1}{|B|^{\frac{1}{t'}}}$
\end{itemize}
For this $a$ we will use the terminology that $a$ is a Hardy-Sobolev $(1,t)$-atom corresponding to the ball $B$.\\
We define the space \HSt{t} as follows: $f\in \HSt{t}$, if there
exists a family of Hardy-Sobolev $(1,t)$-atoms $\{a_j\}_j$ such
that $f$ can be decomposed as
$$f=\sum_j \lambda_ja_j$$
with $\sum_j|\lambda_j|<\infty$. We equip $\HSt{t}$ with the norm
$||f||_{\HSt{t}}=\inf \sum_j|\lambda_j|$, where the infimum is
taken over all possible decompositions.
\end{definition}
Thus $\HSt{t}\subset W^{1,1}$. If one compares this definition
with the Definition 4.1 in \cite{BD09} for non-homogeneous
Hardy-Sobolev $(1,t)$-atoms, one sees that we do not impose the
cancellation condition $\int a=0$ on the atoms. This is due to the
fact that we do want constant functions to belong to our space. On
the other hand our atoms will always satisfy cancellation
condition on the level of derivatives:
$$\int_{\dom}\nabla_T a=0.$$

Moreover if one compares the Definition \ref{DefinitionAtom} with
the Definition 2.11 in \cite{BD09} for homogeneous Hardy-Sobolev
$(1,t)$-atoms one sees that N. Badr and G. Dafni impose
\begin{align}\label{L1Cond}
||a||_{L^1}\leq r(B),
\end{align}
which automatically holds for our atoms, because: For $a$ an atom corresponding to a ball $B$ with $|B|\leq \frac{1}{2}|\partial\Omega|$ we can use Poincar\'e's inequality and the fact that $\nabla a$ is uniformly in $L^1$. In the case that $|B|>\frac{1}{2}|\partial\Omega|$, condition (\ref{L1Cond}) simplifies to $||a||_{L^1}\leq C$, which obviously holds for any atom.\\
\begin{lemma} Let $a$ be a Hardy-Sobolev $(1,t)$-atom, then
$$||a||_{\mathcal C^1}\leq C_t.$$
Thus $\HSt{t}\subset\mathcal C^1$ with $||f||_{\mathcal C^1}\leq
C_t||f||_{\HSt{t}}.$
\end{lemma}
\begin{proof} The proof follows easily from the proof of Proposition 4.5 in \cite{BD09}. \end{proof}
To show the converse, i.e. that $\mathcal C^1\subset \HSt{t}$, we
have to construct the Hardy-Sobolev $(1,t)$-atoms, for which we
will need the following variant of the Calder\'{o}n Zygmund
decomposition:
\begin{theorem}\label{VCZDHS} Let $f\in \mathcal C^1$ and $q$ and $s$ be as in (\ref{6inBD09}). Then for every $\alpha\geq \alpha_0= C_{\Omega} ||f||_{\mathcal C^1}$ with $C_{\Omega}$ a constant depending on the domain $\Omega$, one can find balls $\{B_i\}_i\subset \partial \Omega$, functions $b_i\in W^{1,1}$ and $g\in W^{1,\infty}$ such that
\begin{itemize}
\item $f=g+\sum_i b_i$ \item $|g|+|\nabla g|\leq C \alpha$ almost
everywhere \item $\supp b_i\subset B_i$, $||b_i||_1\leq C r_i
\alpha |B_i|$, $||b_i||_{q}+||\nabla b_i||_q \leq C\alpha
|B_i|^{\frac{1}{q}}$ \item $\sum_i |B_i|\leq \frac{C}{\alpha}\int
(Mf+ Nf)$ \item $\sum_i \chi_{B_i}\leq C$.
\end{itemize}
\end{theorem}
\begin{proof}
The same proof as in Proposition 4.6 of \cite{BD09} works here.
\end{proof}
\begin{theorem}\label{ContinuousAtoms}
Let $f,q$ and $s$ be as in Theorem \ref{VCZDHS} and
$q^*=\frac{sq}{s-q}(>1)$. There exists a family of Hardy-Sobolev
$(1,q^*)$-atoms $\{a_j\}_j$ such that
$$f=\sum \lambda_ja_j\qquad\text{ and }\sum_j |\lambda_j|\leq C||f||_{\mathcal C^1}.$$
Thus $\mathcal C^1\subset \HSt{t}$ for $1<t\leq q^*$.
\end{theorem}
\begin{proof}
The major difference to the proof of Proposition 4.7 in
\cite{BD09} is the fact that our domain is bounded. Let $\alpha_0$
be as in the proof of Theorem \ref{VCZDHS}. Then for every $j\geq
j_0$ with $j_0$ the smallest integer such that $2^{j_0}>\alpha_0$
we apply Theorem \ref{VCZDHS} to get
$$f=g^j + \sum_i b_i^j.$$
Following the proof in \cite{BD09}, we see that we can write
$$f=\sum_{j\geq j_0} (g^{j+1} - g^j) + g^{j_0}$$
in the $W^{1,1}$ sense. The terms $(g^{j+1}-g^j)$ are treated as
in \cite{BD09}. The term $g^{j_0}$ is seen after a normalization
as an atom for $\partial\Omega$. Then one can follow the proof of
the Proposition 4.7 in \cite{BD09} to complete the proof.
\end{proof}
\noindent{\it Remark.} From the construction of the atoms $a_j$ we
see that if $f\in C^0(\partial\Omega)$, then the $a_j$ are in
$C^0(\partial\Omega)$.\vglue2mm

Since in our setting Poincar\'e's inequality on $L^1$ holds and
every $(1,t)$-atom can be decomposed in a $(1,\infty)$-atom and a
$(1,t)$-atom that satisfies the cancellation condition, Theorem
0.1 in \cite{BB10} gives:
\begin{theorem}
$\HSt{t_1} = \HSt{t_2}$ for all $1<t_1,t_2\leq \infty$. The norms
are comparable, where the implicit constant depends on $t_1$ and
$t_2$.
\end{theorem}
Thus we can define $\HS=\HSt{t}$ for any $1<t\leq \infty$ and we will impose the norm of $\HSt{\infty}$ on $\HS$.\\
We finish this section with a result about the $\mathcal
C^q$-norm, which is equivalent to the $\HS$-norm in the $q=1$
case. In order to keep the notation simple, we assume that we work
on $\mathbb R^n$ instead of $\partial \Omega$.
\begin{lemma}\label{MainHSEstimate}
Fix $0<R$ and $0<q\leq 1$. Let $\varphi \in C_0^{\infty}(\mathbb
R^n)$ be supported in $B_{2R}(0)$ with values in $[0,1]$,
$\varphi\equiv 1$ on $B_R(0)$ and $|\nabla \varphi|\leq
\frac{C}{R}$. Assume that $f\in \mathcal C_q^1\cap C^0$ and let
$C_R=\fint_{B_{2R}(0)} f$. Then there exists and $C_0$ independent
of $f$ such that
$$||\varphi(f-C_R)||_{\mathcal C^q}^q\leq C_q R^n M [M(\nabla f)^q](x) + C_q R^n R^q M(|\nabla f|)(x)^q$$
for any $x\in B_{C_0R}(0)$.
\end{lemma}
\begin{proof}
First we claim that for $x\in B_{2R}(0)$ one has
$(\varphi[f-C_R])^b_{q}(x)\leq C M(|\nabla f|)(x)$. For $x\in
B_{2R}(0)$ H\"older's inequality implies
\begin{align*}
(\varphi[f-C_R])^b_{q}(x) &=\sup_{B\ni x}\inf_{c\in \mathbb R}\frac{1}{|B|^{\frac{1}{n}}}\left(\fint_{B} |\varphi(f-C_R)-c|^q\right)^{\frac{1}{q}}\\
&\leq \sup_{B\ni x}\frac{1}{|B|^{\frac{1}{n}}}\fint_{B} |\varphi(f-C_R)-(\varphi[f-C_R])_B|\\
&\leq C \sup_{B\ni x}\fint_B |\nabla (\varphi[f-C_R])|\\
&\leq C \sup_{B\ni x}\fint_B |\nabla \varphi||f-C_R|+\sup_{B\in x}\fint_B\varphi |\nabla f|\\
&\leq \frac{C}{R}\sup_{\stackrel{B\ni x}{r(B)>R}}\fint_B\chi_{B_{2R}} |f-C_R| + \frac{C}{R} \sup_{\stackrel{B\ni x}{r(B)\leq R}}\fint_B\chi_{B_{2R}} |f-C_R|\\
&\quad+ M(|\nabla f|)(x)\\
&=I+ II + M(|\nabla f|)(x).
\end{align*}
For $I$ observe that $B\cap B_{2R}(0)\neq \emptyset$ implies
$B_{2R}(0)\subset 5B$ and so
\begin{align*}
I \leq \frac{C}{R}\sup_{\stackrel{B\ni x}{r(B)>R}} \frac{1}{|B|}
\int_{B_{2R}}|f-C_R| \leq \frac{C}{|B_R|}\int_{B_{2R}}|\nabla f|
\leq M(|\nabla f|)(x).
\end{align*}
For $II$, we first use the fact that the uncentered maximal
function is dominated by $c_n$ times the centered dyadic maximal
function. Hence it is enough to consider for the supremum balls of
the form $B_j=B(x,R2^{-j+1}),j\geq 0$:
\begin{align*}
II & \leq \frac{C}{R} \sup_{j\geq 0} \fint_{B_j} |f-C_R|\\
&\leq \frac{C}{R} \sup_{j\geq 0} \fint_{B_j} |f-f_{B_j}| + \frac{C}{R}\sum_{j\geq 0} |f_{B_{j+1}}-f_{B_j}|\\
&\leq \frac{C}{R}\sup_{j\geq 0} 2^{-j}R\fint_{B_j}|\nabla f| + \frac{C}{R} \sum_{j\geq 0}\fint_{B_{j+1}}|f-f_{B_j}|\\
&\leq C M(|\nabla f|)(x) + \frac{C}{R} \sum_{j\geq 0} 2^{-j}R \fint_{B_j}|\nabla f|\\
&\leq C M(|\nabla f|)(x)
\end{align*}
i.e. the claim is proved. To use the claim we write
$$||(\varphi[f-C_R])_q^b||^q_q = \int_{B_{2R}(0)} \left((\varphi[f-C_R])_{q}^b\right)^q+ \int_{B_{2R}(0)^c} \left((\varphi[f-C_R])_{q}^b\right)^q.$$
By the previous claim the first term is bounded by $ C
R^nM[M(|\nabla f|)^q](x)$ for any $x\in B_{C_0R}(0)$. For the
second term we will use the fact that if $x\in B$ and $|x|\approx
2^jR$ then for $B\cap B_{2R}\neq \emptyset$ one needs $r(B)\geq C
2^j R$. Thus we have
\begin{align*}
\int_{B_{2R}(0)^c}\left([\varphi(f-C_R)]_{q}^b\right)^q&=\sum_{j\geq 1}\int_{\{|x|\approx 2^jR\}} \left[ \sup_{B\ni x}\inf_{c\in \mathbb R} \frac{1}{|B|^{\frac{1}{n}}}\left(\fint_B |\varphi(f-C_R)-c|^q\right)^{\frac{1}{q}}\right]^q\diff x\\
\{\text{choose $c =0$} \}\quad &\leq\sum_{j\geq 1}\int_{\{|x|\approx 2^jR\}}\left[\sup_{B\ni x}\frac{1}{|B|^{\frac{1}{n}}}\left(\fint_B \chi_{B_{2R}} |f-C_R|^q\right)^{\frac{1}{q}}\right]^q\\
&\leq C \sum_{j\geq 1}\int_{\{|x|\approx 2^jR\}}\left[\frac{1}{2^jR}\left(\frac{1}{(2^{j}R)^n}\int_{B_{2R}} |f-C_R|^q\right)^{\frac{1}{q}}\right]^q\diff x\\
&\leq C \sum_{j\geq 1} (2^jR)^n\frac{1}{(2^jR)^q}\frac{1}{(2^jR)^n}\left(\int_{B_{2R}} |f-C_R|^q\right)\\
&\leq C \sum_{j\geq 1}\frac{1}{(2^jR)^q}\left(\int_{B_{2R}}|f-C_R|\right)^q  |B_{2R}|^{1-q}\\
\quad &\leq C_q \left(\int_{B_{2R}}|\nabla f|\right)^q R^{n(1-q)}\\
&\leq C_q R^nM(|\nabla f|)(x)^q
\end{align*}
for any $x\in B_{C_0R}(0)$.\\
To deal with the $L^q$-norm of $\varphi(f-C_R)$ one applies
H\"older's inequality and Poincar\'e's inequality to get
$||\varphi(f-C_R)||_{L^q}\leq C R^n R^q M(|\nabla f|)(x)^q$ for
any $x\in B_{C_0R}$. Thus the proof of the Lemma is complete.
\end{proof}

\section{The Regularity Problem for boundary data in $\HS$}
We start this section by adjusting some results from \cite{KP93}
to the $(R)_{\HS}$-case. By the proof of Theorem 3.1 in
\cite{KP93} and the Vitali-Hahn-Soks Theorem (see for example
\cite{Doo94}, p.155) we get for
$$(\nabla_Tu)_r(Q)=\fint_{B_{r/2}(A_r(Q))}\nabla u(X)\cdot \vec{T}(Q)\diff X$$
\begin{theorem}\label{HSKP9331}
Assume that $u\in W_{loc}^{1,2}(\Omega)$ solves $Lu=0$ and that
$||\tilN{u}||_{L^1(\partial \Omega)} + ||u||_{L^1(\Omega)}<\infty$
then
\begin{itemize}
\item $u$ converges non-tangentially almost everywhere to a
function $f$ with $f\in W^{1,1}(\partial \Omega)$. \item If $f=0$
almost everywhere, then $u\equiv 0$. \item There exists a sequence $r_j\rightarrow 0$ such that $(\nabla_Tu)_{r_j}$ converges in the weak$^*$ topology of
$(L^{\infty}(\partial \Omega))^*$ to $\nabla_Tf$.
\end{itemize}
\end{theorem}

We first observe that the solvability of $(R)_{\HS}$ can be
reduced to proving the estimate (\ref{RHSATO}) for smooth atoms.

\begin{lemma}\label{RpSmoothEnough} Assume that (\ref{RHSATO}) holds for smooth Hardy-Sobolev atoms, then $(R)_{\HS}$ holds.
\end{lemma}

\begin{proof} We first claim that if (\ref{RHSATO}) holds for all continuous
Hardy-Sobolev atoms then $(R)_{\HS}$ holds. Indeed, let $f\in
\HS\cap C^0(\partial \Omega)$. Then by the remark below Theorem
\ref{ContinuousAtoms} there exist continuous atoms $a_j$ and
scalars $\lambda_j$ such that $f=\sum \lambda_ja_j$. Thus if $u$
is the solution for $f$ and $u_j$ for $a_j$ we have
\begin{align*}
& ||\tilN{u}||_{L^1(\partial \Omega)} \leq \sum_j |\lambda_j| \; ||\tilN{u_j}||_{L^1(\partial \Omega)} \leq C  \sum_j|\lambda_j|,\\
& ||u||_{L^1(\Omega)} \leq
\sum_j|\lambda_j|\;||u_j||_{L^1(\Omega)} \leq C \sum_j|\lambda_j|.
\end{align*}
Since this holds for all decompositions we get $||\tilN{u}||_{L^1(\partial\Omega)} + ||u||_{L^1(\Omega)}\leq C ||f||_{\HS}$ and so the claim holds. Hence it is enough to prove (\ref{RHSATO}) for continuous Hardy-Sobolev atoms $a$ under the assumption that (\ref{RHSATO}) holds for smooth Hardy-Sobolev atoms.\\
Every continuous Hardy-Sobolev atom $a$ can be uniformly
approximated in $\HS$ by smooth Hardy-Sobolev atoms $a_j$ (by the
use of mollifiers). We call the corresponding weak solutions $u$
and $u_j$. The maximum principle implies that $u_j$ converges
uniformly to $u$ on $\bar{\Omega}$, hence $||u||_{L^1(\Omega)}\leq
\lim_j ||u_j||_{L^1(\Omega)}\leq C \lim_j ||a_j||_{\HS}\leq C
||a||_{\HS}$. Let
$$N_{\varepsilon}(h)(Q)=\sup_{\stackrel{X\in \Gamma(Q)}{\delta(X)\geq \varepsilon}} \left(\fint_{B(X,\delta(X)/2)} |h|^2\right)^{\frac{1}{2}}$$
be the truncated below maximal function. Cacciopoli's inequality
and the uniform convergence of $u_j$ to $u$ imply
$N_{\varepsilon}(\nabla u_j - \nabla u)\rightarrow 0$ uniformly on
$\dom$. Therefore
$$\int_{\partial \Omega} N_{\varepsilon}(\nabla u) \leq \lim_{j\rightarrow \infty} \int_{\partial \Omega}N_{\varepsilon}(\nabla u_j)\leq C\lim_j||a_j||_{\HS}\leq C ||a||_{\HS}.$$
Since $N_{\varepsilon}$ increases to $N$ the monotone convergence
theorem completes the proof.
\end{proof}

Recall that when we defined the $(R)_{\HS}$ solvability we only
did it for data in $\HS(\dom) \cap C^0(\dom)$. The following
theorem shows that this is sufficient and that this implies
existence of a unique solution for any data in $\HS(\dom)$.

\begin{theorem}
Assume that $(R)_{\HS}$ holds. Given $f\in \HS$, there exists a
unique $u\in L^1(\Omega)$ with $N(\nabla u)\in L^1(\partial
\Omega)$ such that $Lu=0$ in $\Omega$ and $u$ converges
non-tangentially almost everywhere to $f$. Moreover $(\nabla_T
u)_{r_j}$ ($r_j\to 0$) converges in the weak$^*$ topology of
$(L^{\infty}(\Omega))^*$ to $\nabla_Tf$.
\end{theorem}
\begin{proof}
We have seen that the norms of $\HSt{t_1}$ and $\HSt{t_2}$ for
$1<t_1,t_2\leq \infty$ are equivalent. Thus every
$(1,\infty)$-atom can be approximated by smooth $(1,\infty)$-atoms
in $\HS$. Let $f=\sum_j\lambda_ja_j$, then choose smooth
$(1,\infty)$-atoms $a_j^N$ with $||a_j^N-a_j||_{\HS}\leq
\varepsilon\frac{1}{2^j}\frac{1}{\sum_j|\lambda_j|}$. Now choose
$N$ such that $\sum_{j>N}^{\infty}|\lambda_j|\leq \varepsilon$,
then for $f^N=\sum_{j=1}^N\lambda_ja_j^N$ we have
$||f-f_N||_{\HS}\leq \sum_{j=1}^N |\lambda_j| ||a_j^N-a_j||_{\HS}
+ \sum_{j>N}^{\infty}|\lambda_j|\leq 2\varepsilon$, i.e.
$f_N\rightarrow f$ in $\HS$ with $f_N$ smooth.

If follows that we can choose $f_j\in \HS\cap
C^{\infty}(\partial\Omega)$ converging to $f\in \HS$ in $\HS$
norm. Denote by $u_j$ the weak solution for the smooth boundary
data $f_j$. Then
$$||\tilN{(u_j-u_k)}||_{L^1(\dom)} + ||u_j-u_k||_{L^1(\Omega)}\rightarrow 0,$$
and so $\{u_j\}_j$ is a Cauchy sequence in $L^1(\Omega)$. Thus
there exists $u\in L^1(\Omega)$ such that $u_j\rightarrow u$ in
$L^1(\Omega)$. Using Cacciopoli's inequality in the interior we
see that for any compact $K\subset\Omega$ one has
$$||u_j-u_k||_{W^{1,2}(K)}\leq C_K ||u_j-u_k||_{L^1(K)}\rightarrow 0.$$
The uniqueness of limits implies that $u\in
W^{1,2}_{\loc}(\Omega)$ and that $u$ is a weak solution of the
equation $Lu=0$. Furthermore
\begin{align*}
& ||u||_{L^1(\Omega)}  =\lim_{j\rightarrow \infty} ||u_j||_{L^1(\Omega)}\leq C\lim_j ||f_j||_{\HS}\leq C||f||_{\HS},\\
& ||u-u_j||_{L^1(\Omega)} \leq C ||f-f_j||_{\HS}.
\end{align*}
By using the same $N_{\varepsilon}$-idea as before we get
\begin{align*}
& ||\tilN{u}||_{L^1(\partial \Omega)} \leq C||f||_{\HS},\\
& ||\tilN{(u-u_j)}||_{L^1(\partial \Omega)} \leq C||f-f_j||_{\HS}.
\end{align*}
Hence Theorem \ref{HSKP9331} implies that $u$ has a non-tangential
limit almost everywhere, which we will denote by
$u|_{\partial\Omega}$. It remains to check that $u|_{\partial
\Omega} = f$ almost everywhere. We know that
$$(u_j - u)^*(Q)\leq C \tilN{(u-u_j)}(Q) + C ||u_j-u||_{L^1(\Omega)}.$$
Therefore
\begin{align*}|\{|f-u|_{\partial \Omega}|>\alpha\}| &\leq |\{|f-f_j|>\textstyle\frac{\alpha}{3}\}| + |\{|f_j-u_j|>\frac{\alpha}{3}\}| + |\{|u_j-u|_{\partial\Omega}|>\frac{\alpha}{3}\}|\\
&\leq \textstyle\frac{C}{\alpha} ||f-f_j||_{L^1(\partial\Omega)} + |\{(u_j-u)^*\geq \frac{\alpha}{3}\}|\\
&\leq \textstyle\frac{C}{\alpha} \left(||f-f_j||_{L^1(\partial\Omega)} +||\tilN{(u_j - \nabla u)}||_{L^1(\partial\Omega)} + ||u_j - u||_{L^1(\Omega)}\right)\\
&\leq \textstyle\frac{C}{\alpha}||f-f_j||_{\HS}
\end{align*}
which implies the non-tangential convergence almost everywhere.
Uniqueness and the stated $(\nabla_T u)_{r_j}$ convergence follow
from Theorem \ref{HSKP9331}, which completes the proof.
\end{proof}

\subsection{$(R)_{\HS}$ implies $(D^*)_{\text{BMO}}$} In this
subsection we explore the relation between the $(R)_{\HS}$ and the
elliptic measure of the adjoint operator $L^*$.

Let us recall the definition of the elliptic measure. In
\cite{LSW63} it was proved that for every $g\in
C^0(\partial\Omega)$ there exists a unique $u\in
W^{1,2}_{loc}(\Omega)\cap C^0(\bar{\Omega})$ such that $L^*u=0$ in
$\Omega$ and $u=g$ on $\partial \Omega$. By the maximum principle
we have $||u||_{L^{\infty}(\Omega)}\leq
||g||_{L^{\infty}(\Omega)}$. Thus for every fixed $X\in\Omega$ the
map defined by
$$C^0(\partial \Omega)\ni g\mapsto u(X)$$
is a bounded linear functional on $C^0(\partial \Omega)$. The
Riesz Representation Theorem implies the existence of a unique
regular Borel measure $\omega^X$ such that
$$u(X)= \int_{\partial \Omega} g(Q)\diff \omega^X(Q).$$
We will write $\omega$ instead of $\omega^{X_0}$ if we speak about
a fixed $X_0$. The reverse H\"older class $B_q$, $q>1$, is defined
as the class of all non-negative functions $k\in L^1_{\loc}$ such
that
$$\left(\fint_Q k^q\right)^{\frac{1}{q}}\leq C\fint_Q k$$
for all cubes $Q$, where $\fint_Q k=\frac{1}{|Q|}\int_Q k$. Using
for example Lemma 1.4.2 in \cite{Ken94} one sees that
$$(D^*)_p \Leftrightarrow \omega\in B_{p'}(d\sigma).$$
By the result of \cite{DKP09} we also have
$$(D^*)_{BMO} \Leftrightarrow \omega\in A_\infty(d\sigma)=\bigcup_{p'>1}B_{p'}(d\sigma).$$

Let us recall a variant of the non-tangential maximal function
from \cite{KP93}. For any $h:\Omega\to \mathbb R$, $Q\in
\partial\Omega$ we consider
$S_{\varepsilon,R}(Q)=T_R(Q)\cap(\partial \Omega)_{\varepsilon R}$
and define
$$N^{\varepsilon}(h)(Q)=\sup_{X\in\Gamma(Q)} \left( \fint_{T_{\delta(X)}(\hat{X})\backslash S_{\varepsilon,\delta(X)}(\hat{X})}|\nabla h(Z)|^2\diff Z\right)^{\frac{1}{2}}.$$
\begin{lemma} For all $0<p<\infty$ there exists $C_1,C_2$ depending only on $\varepsilon$, $p$ and $\Omega$ such that
$$C_1||N^{\varepsilon}(h)||_{L^p(\partial \Omega)}\leq ||N(h)||_{L^p(\partial \Omega)}\leq C_2 ||N^{\varepsilon}(h)||_{L^p(\partial \Omega)}.$$
\end{lemma}
\begin{proof} As it is stated in \cite{KP93}, the proof can be found in \cite{FS72}, Lemma 1, Section 7.
\end{proof}
\begin{lemma}[Lemma 5.8 and Lemma 5.13 in \cite{KP93}]\label{58513} Let $0<R<\frac{1}{4} R'$ and $Q\in\partial\Omega$. Assume that $u$ is a non-negative weak solution, which vanishes on $\Delta_{R'}(Q)$, then there exists an $\varepsilon>0$ such that
$$\int_{T_R(Q)} |\nabla u|^2\leq C\int_{T_R(Q)\backslash S_{\varepsilon,R}(Q)} |\nabla u|^2.$$
Moreover for $X\in T_{\frac{1}{4}R'}(Q)$ and $\delta(X)=R$ we have
$$\frac{u(X)}{\delta(X)}\approx\left(\fint_{T_R(Q)\backslash S_{\varepsilon,R}(Q)}|\nabla u|^2\right)^{\frac{1}{2}}.$$
\end{lemma}
\begin{theorem}\label{RpToDp}
$(R)_{\HS}$ implies $(D^*)_{\text{BMO}}$ (and also $(D^*)_p$ for
some $1<p<\infty$).
\end{theorem}
\begin{proof}
We use the methods and ideas from the proof in \cite{KP93} and change them a bit to suit the $(R)_{\HS}$ condition. Let $\omega$ be the elliptic
measure for $L^*$. By \cite{DKP09} it suffices to prove that $\omega$ is absolutely continuous with respect to the surface measure and that
$\omega\in A_{\infty}(d\sigma)$.\\
Choose $R\leq \frac{1}{5}R_0$ and
$Q_0\in\partial\Omega$. Let $f\in C^{\infty}(\partial \Omega)$ be
non-negative with $0\leq f\leq 1$, $|\nabla f|\le \frac{C}{R}$ and
\begin{align*}
\begin{cases}
f\equiv 0 &\text{on }\Delta_R=\Delta_R(Q_0)\\
f\equiv 1 &\text{on }\Delta_{4R}\backslash \Delta_{2R}\\
f\equiv 0 &\text{on }\partial \Omega\backslash \Delta_{5R}.
\end{cases}
\end{align*}
Clearly, $||f||_{L^{\infty}(\partial \Omega)}\leq 1$ and $||\nabla
f||_{L^{\infty}(\partial\Omega)}\leq \frac{C}{R}$, thus
$\frac{C}{R^{n-2}}f$ is a Hardy-Sobolev $(1,\infty)$-atom. If
follows that $||f||_{\HS}\leq C R^{n-2}$.

Let $u$ be the weak solution with boundary data $f$. Then $C\leq
u(A_R(Q_0))\leq 1$. By the comparison principle and Lemma 2.2 in
\cite{CFMS81} we have for $X\in T_{R/2}(Q_0)$:
\begin{align*}
\frac{u(X)}{G(X,0)}\approx
\frac{u(A_R(Q_0))}{G(A_R(Q_0),0)}\approx
\frac{1}{G^*(0,A_R(Q_0))}\approx \frac{R^{n-2}}{\omega(\Delta_R)}.
\end{align*}
Lemma \ref{58513} and Lemma 2.2 in \cite{CFMS81} imply
$$\left(\fint_{T_{\delta(X)}(\hat{X})\backslash S_{\varepsilon,\delta(X)}(\hat{X})}|\nabla u|^2\right)^{\frac{1}{2}}
\approx \frac{u(X)}{\delta(X)}\approx
\frac{G(X,0)}{\delta(X)}\frac{R^{n-2}}{\omega(\Delta_R)}\approx
\frac{\omega(\Delta_{\delta(X)}(\hat{X}))}{\delta(X)^{n-1}}\frac{R^{n-2}}{\omega(\Delta_R)}$$
and so for $P=\hat{X}$ we have
\begin{align*}
\frac{\omega(\Delta_{\delta(X)}(P))}{\delta(X)^{n-1}}\approx
\frac{\omega(\Delta_R)}{R^{n-2}}
\left(\fint_{T_{\delta(X)}(\hat{X})\backslash
S_{\varepsilon,\delta(X)}}|\nabla
u|^2\right)^{\frac{1}{2}}\leq\frac{C\omega(\Delta_R)}{R^{n-2}}
N^{\varepsilon}(\nabla u)(P).
\end{align*}
Hence if we define
$h(P)=\sup_{0<s<\frac{R}{2}}\frac{\omega(\Delta_s(P))}{s^{n-1}}$,
the estimate above gives that $h(P)\leq
\frac{C\omega(\Delta_R)}{R^{n-2}}N^{\varepsilon}(\nabla u)(P)$. By
Lemma \ref{58513}, the assumption that $(R)_{\HS}$ holds and the
doubling property of $\omega$ we see that $\omega$ is absolutely
continuous with respect to $\diff\sigma$, i.e. $\omega=k\diff
\sigma$ for some $k\in L^1(\diff\sigma)$.

In order to show that $\omega\in A_{\infty}(\diff\sigma)$ it is
enough to show that $||\omega||_{L(log
L)(\diff\tilde{\sigma})}\leq C
||\omega||_{L^1(\diff\tilde{\sigma})}$ for all
$\diff\tilde{\sigma}= \frac{\chi_{\Delta}}{|\Delta|}\diff\sigma$
(see for example \cite{Fef89}), where we can assume without
loosing generality that $r(\Delta)\leq R_0$. We have by
\cite{Ste69} that
\begin{align*}
||k||_{L(logL)(\diff\tilde{\sigma})}\leq C ||M_{\Delta} k
||_{L^1(\diff\tilde{\sigma})},
\end{align*}
where $M_{\Delta}$ denotes the Hardy-Littlewood maximal function
over all balls contained in $\Delta$. By the doubling property of
$\omega$ we see that
\begin{align*}
||M_{\Delta}k||_{L^1(\diff\tilde{\sigma})} & \leq C \fint_{\Delta_{R/2}}h(P)\,d\sigma(P)\\
& \leq \frac{C\omega(\Delta_R)}{R^{n-2}}\fint_{\Delta_{R/2}}N^{\varepsilon}(\nabla u)(P)\,d\sigma\\
& \leq \frac{C\omega(\Delta_R)}{R^{n-2}} \frac{1}{R^{n-1}} R^{n-2}
= C||\omega||_{L^1(\diff\tilde{\sigma})},
\end{align*}
which concludes that $\omega\in A_{\infty}(d\sigma)$ proving our
claim.
\end{proof}

\subsection{A new proof for: $(R)_p$ implies $(R)_{\HS}$}

In \cite{KP93} C.E. Kenig and J. Pipher used localization argument
to prove the implication that $(R)_p$ implies $(R)_{\HS}$. In
order to prove the same result without the localization theorem of
\cite{KP93} we need the following:
\begin{lemma}[Lemma 2.5 in \cite{She07}]\label{NothingSpecial}
Let $u$ be a weak solution for $L$ in $\Omega$ which vanishes on $\Delta_{5R}(Q)$.
Then for any $X\in T_{2R}(Q)$ we have
$$|u(X)|\approx \frac{G(X,0)}{G(A_R(Q),0)} \left(\fint_{T_{4R}(Q)}|u|^2\right)^{\frac{1}{2}}.$$
\end{lemma}
The next Lemma is part of the proof of Theorem 2.9 in
\cite{She07}:
\begin{lemma}\label{ReversePartI} Assume that $\omega\in A_{\infty}(d\sigma)$. Then for $u$ and $R$ as in Lemma \ref{NothingSpecial} we get
$$\left(\fint_{\Delta_R(Q)}\left(\frac{u}{\delta}\right)_R^*\right)\leq \frac{C}{R}\left(\fint_{T_{4R}(Q)} |u|^2\right)^{\frac{1}{2}}.$$
\end{lemma}
\begin{proof}
By Lemma \ref{NothingSpecial} we have for any $P\in \Delta_R(Q)$
$$ \left(\frac{u}{\delta}\right)^*_R(P)\leq C \frac{1}{G(A_R,0)}
\left(\fint_{T_{4R}(Q)}|u|^2\right)^{\frac{1}{2}}\left(\frac{G(\cdot,0)}{\delta(\cdot)}\right)^*_R(P).$$
Lemma 2.2 in \cite{CFMS81} and (1.3) Theorem in \cite{GW82} imply
$$\frac{G(X,0)}{\delta(X)} \approx \frac{\omega(\Delta_{\delta(X)}(\hat{X}))}{\delta(X)^{n-1}}.$$
Thus for $h_R(Q)=\sup_{X\in\Gamma(Q)\cap
B_R(Q)}\frac{\omega(\Delta_{\delta(X)}(\hat{X}))}{\delta(X)^{n-1}}$
we get
\begin{align*}
\left(\fint_{\Delta_R(Q)}\left(\frac{u}{\delta}\right)^*\right) &\leq \frac{C}{G(A_R(Q),0)}
\left(\fint_{T_{4R}(Q)} |u|^2\right)^{\frac{1}{2}} \left( \fint_{\Delta_R(Q)} h_R\right)\\
&\leq \frac{C}{G(A_R(Q),0)} \left( \fint_{T_{4R}(Q)}
|u|^2\right)^{\frac{1}{2}} \frac{\omega(\Delta_R(Q))}{R^{n-1}},
\end{align*}
where for the last step we used the $A_{\infty}(d\sigma)$
condition. Thus
$$\left(\fint_{\Delta_R(Q)} \left(\frac{u}{\delta}\right)^*\right)\leq \frac{C}{R}\left(\fint_{T_{4R}(Q)} |u|^2\right)^{\frac{1}{2}}.$$
\end{proof}

The result below takes care of the estimate for non-tangential
maximal function away from the support of an $(1,\infty)$-atom.

\begin{theorem}\label{ZeroPart} Assume that $\omega\in A_{\infty}(d\sigma)$, where $\omega$ is the elliptic measure of the operator $L^*$.
Let $f$ be a smooth Hardy-Sobolev $(1,\infty)$-atom
corresponding to the surface ball $\Delta_R(Q_0)$. Then $u$ the
weak solution for $L$ with boundary data $f$ satisfies
$$||\tilN{u}||_{L^1(\partial\Omega\backslash\Delta_{8R}(Q_0))}\leq C$$
for a constant $C$ independent of $f$ and $R$.
\end{theorem}
\begin{proof}
Without loosing generality we can assume that $R\leq R_0$ and that
$f$ is non-negative. Since $f$ is a smooth Hardy-Sobolev
$(1,\infty)$-atom for $\Delta_R(Q)$ we have $|f|\leq
\frac{C}{R^{n-2}}$. Thus for $X\in \Omega\backslash T_{2R}(Q)$
Lemma 2.2 in \cite{CFMS81} and Theorem 1.8 in \cite{GW82} imply
\begin{align}\label{EssentialDecay}
u(X)\leq C R^{2-n} w^X(\Delta_R(Q_0))\approx G(X,A_R(Q_0))\leq
C\frac{R^{\alpha}}{|X-Q_0|^{n+\alpha -2}}.
\end{align}
Define $R_j=\{ Q\in\partial\Omega:|Q-Q_0|\approx 2^jR\}$ for
$j\geq 3$. For $Q\in R_j$ and $X\in \Gamma(Q)$ with $|X-Q|\geq
2^jR$ we have by (\ref{EssentialDecay}) and Cacciopoli's
inequality
$$\left(\fint_{B_{\frac{\delta(X)}{2}}(X)}|\nabla u|^2\right)^{\frac{1}{2}}\leq \frac{C}{\delta(X)} \left(\fint_{B_{\frac{\delta(X)}{2}}(X)} u^2\right)^{\frac{1}{2}}\leq \frac{C}{\delta(X)}\frac{R^{\alpha}}{(2^jR)^{n+\alpha -2}}\leq\frac{C}{2^{j\alpha}}\frac{1}{(2^jR)^{n-1}}.$$
Therefore
$$\int_{R_j}\tilN{u}(Q)\,d\sigma(Q)\leq \int_{R_j}N_{2^jR}(\nabla u)(Q)\, d\sigma (Q)+ \frac{C}{2^{j\alpha}},$$
where $N_{2^jR}$ is as before the truncated non-tangential maximal
function at the height $2^jR$. By Cacciopoli inequality in the
interior we get
$$\int_{R_j}N_{2^jR}(\nabla u)(Q)\,d\sigma (Q)\leq C \int_{R_j}\left(\frac{u}{\delta}\right)^*(Q)\,d\sigma(Q).$$
Thus if we cover $R_j$ with finite many balls $\Delta_{\alpha}^j$
with radius comparable to $2^jR$ and apply Lemma
\ref{ReversePartI} to each of the balls we get
$$\int_{R_j}N_{2^jR}(\nabla u)(Q)\,d\sigma(Q)\leq C (2^jR)^{n-1}\sum_{\alpha}\fint_{\Delta_{\alpha}^j}
\left(\frac{u}{\delta}\right)^* \leq
\frac{C(2^jR)^{n-1}}{2^jR}\sum_{\alpha}
\left(\fint_{T_{\Delta_{\alpha}^j}} u^2\right)^{\frac{1}{2}},$$
where $T_{\Delta_{\alpha}^j}=T_{r_{\alpha}^j}(Q_{\alpha}^j)$ for
$r_{\alpha}^j=r(\Delta_{\alpha}^j)$ and $Q_{\alpha}^j$ the center
of $\Delta_{\alpha}^j$. Inequality (\ref{EssentialDecay}) implies
that each term is bounded by
$\frac{R^{\alpha}}{(2^jR)^{n+\alpha-2}}$, thus
$$\int_{R_j}N_{2^jR}(\nabla u)(Q)\,d\sigma(Q) \leq C (2^jR)^{n-2}\frac{R^{\alpha}}{(2^jR)^{n+\alpha - 2}}\leq \frac{C}{2^{j\alpha}}.$$
Therefore $\int_{R_j}\tilN{u}\,d\sigma\leq \frac{C}{2^{j\alpha}},$
which means that we can take the sum in $j$ to get
$$\int_{\partial\Omega\backslash \Delta_{8R}(Q)}\tilN{u}\,d\sigma\leq C.$$
\end{proof}
Thanks to Theorem \ref{ZeroPart} we now can reprove Theorem 5.2 of
\cite{KP93}.
\begin{theorem}\label{RpToRHS}
$(R)_p$ implies $(R)_{\HS}$.
\end{theorem}
\begin{proof}
By Lemma \ref{RpSmoothEnough} it is enough to show that
(\ref{RHSATO}) holds for smooth Hardy-Sobolev $(1,\infty)$-atoms.
Let $f$ be a smooth Hardy-Sobolev $(1,\infty)$-atom corresponding
to $\Delta_R(Q)$ and $u$ the weak solution for $f$. Without
loosing generality we can assume that $R\leq R_0$. Either by
slightly adjusting proof for Theorem \ref{RpToDp} (or by Theorem
5.4 \cite{KP93}) we know that $(D^*)_p$ holds and therefore the
elliptic measure $\omega$ of the operator $L^*$ belongs to
$A_\infty(d\sigma)$. From this by Theorem \ref{ZeroPart} we obtain
$$||\tilN{u}||_{L^1(\partial\Omega\backslash \Delta_{8R}(Q))}\leq C.$$
For the $\Delta_{8R}(Q)$ part, we use H\"older's inequality and
the $(R)_p$ condition to get
$$||\tilN{u}||_{L^1(\Delta_{8R}(Q))}\leq C |\Delta_R(Q)|^{\frac{1}{p'}}\: ||\tilN{u}||_{L^p(\Delta_{8R}(Q))}
\leq C
|\Delta_R(Q)|^{\frac{1}{p'}}\:||f||_{H^{1,p}(\partial\Omega)}\leq
C,$$
since $f$ is a $(1,\infty)$-atom for $\Delta_{R}(Q)$.\\
It remains to show that $||u||_{L^1(\Omega)}\leq C$. From
(\ref{EssentialDecay}) we see that for $X\in \Omega\backslash
T_{2R}(Q)$ we have
$$u(X)\leq C G(X,A_R(Q))$$
and so
$$||u||_{L^1(\Omega)}\leq C ||u||_{L^1(\Omega_{R_0})} + ||\tilN{u}||_{L^1(\partial \Omega)}\leq C,$$
which completes the proof.
\end{proof}

\subsection{$(R)_{\HS}$ implies $(R)_p$ for some $1<p<\infty$}

We are now ready to establish the main result of this paper,
namely the implication that $(R)_{\HS}$ implies $(R)_p$ for some
$1<p<\infty$. In the course of thinking about this problem we
discovered that there are two possible ways to establish this
result. One is to adapt the proof in \cite{KP93} where for $(R)_p$
implies $(R)_{p+\varepsilon}$ was established. The other way is
motivated by the proof of the main Theorem in \cite{She07}
(adjusted with the aid of Lemma \ref{MainHSEstimate}). We decided
we prefer the second method as it avoids the use of a localization
theorem and real variable techniques with rather lengthy proofs.
We present this method here.

We define
$$E(\lambda)=\{P\in \partial\Omega: M(\tilN{u})(P)>\lambda\}.$$
\begin{theorem}
Assume that $(R)_{\HS}$ holds. Choose any $p\in (1,\infty)$ for
which the $(D^*)_{p'}$ holds. Let $f\in C^{\infty}(\partial
\Omega)$ and $u$ be the corresponding weak solution of the
Dirichlet problem. Then there exist positive constants
$\varepsilon,\eta, C_0$ such that
\begin{align}\label{HSGoodLambda}
|E(A\lambda)|\leq \varepsilon^{1+\eta} |E(\lambda)| + |\{P\in
\partial\Omega: M(M(|\nabla f|))>\gamma\lambda\}|
\end{align}
for all $\lambda
>\lambda_0=C_0\int_{\partial\Omega}\tilN{u}\,d\sigma$,
$\gamma=\gamma(\varepsilon)$ and $A=\varepsilon^{-\frac{1}{p}}$.
\end{theorem}
\begin{proof}
This proof for the $(R)_p$ case can be found in Lemma 3.4 in
\cite{She07}. The weak $(1,1)$ inequality for the Hardy Littlewood
maximal function implies
$$E(\lambda)\leq \frac{C}{\lambda}\int_{\partial\Omega}\tilN{u}\,d\sigma\leq \frac{C}{\lambda}\frac{\lambda_0}{C_0}.$$
Thus by choosing $C_0=C_0(\Omega)$ sufficiently large we can
ensure that $E(\lambda)\leq \frac{1}{2}|\Delta_{{R_0}/{4}}|$,
where $\Delta_{{R_0}/{4}}$ is any surface ball with radius
$R_0/4$. Thus $E(\lambda)^c\cap \Delta_{{R_0}/{4}}\neq \emptyset$
for $\lambda >\lambda_0$.\vglue1mm

Let $\{Q_k\}$ be a Whitney decomposition of $E(\lambda)$, i.e.
\begin{itemize}
\item $E(\lambda)=\bigcup_kQ_k$ \item $\sum_k\chi_{Q_k}\leq K$
\item $3Q_k\cap E(\lambda)^c\neq \emptyset$.
\end{itemize}\vglue1mm

To prove the lemma it is sufficient to prove that
\begin{equation}
Q_k\cap \{M(M(|\nabla f|))\leq
\gamma\lambda\} \neq \emptyset\text{ implies } |E(A\lambda)\cap
Q_k|\leq \varepsilon^{1+\eta}|Q_k|.\label{GL}
\end{equation}
Indeed, since $E(A\lambda)\subset E(\lambda)$ it follows that for
$\varepsilon$ small enough such that $K\varepsilon^{1+\eta}\leq
\varepsilon^{1+\frac{\eta}{2}}$ we have
\begin{align*}
|E(A\lambda)| & \leq \sum_{\{k:Q_k\cap \{M_{{R_0}/{2}}(M_{R_0}(|\nabla f|))\leq \gamma\lambda\} \neq \emptyset\}} |E(A\lambda)\cap Q_k| + |\{M(M(|\nabla f|))\geq \gamma\lambda\}|\\
&\leq \varepsilon^{1+\frac{\eta}{2}}|E(\lambda)| + |\{M(M(|\nabla
f|))\geq \gamma\lambda\}|,
\end{align*}
which is the statement of our theorem.\vglue1mm

Hence we focus on establishing (\ref{GL}). By the properties
imposed from the Whitney decomposition on $Q_k$ we have for $P\in
Q_k$:
$$M(\tilN{u})(P)\leq \max \{M_{5Q_k}(\tilN{u}),C_1\lambda\}$$
for some $C_1=C_1(\Omega)$ depending only on the geometry of our
domain. Here $M_Q$ is a modified version of the maximal function
$$M_Q(f)(P)=\sup_{\stackrel{\widetilde{Q}\ni P}{\widetilde{Q}\subset Q}}
\fint_{\widetilde{Q}} |f|.$$

Take now $A$ larger than $C_1$ we see by the properties of the
Whitney decomposition on $Q_k$ that
\begin{align}\label{ALargeEnough}|Q_k\cap E(A\lambda)|\leq |\{ P\in Q_k: M_{5Q_k}(\tilN{u})(P)>A\lambda\}|.
\end{align}
Let $v$ be a weak solution to the Dirichlet problem for the
operator $L$ in the domain $\Omega$ with boundary data
$\varphi(f-\alpha)$, where $\varphi\in C^{\infty}(\partial\Omega)$
with $0\leq \varphi\leq 1$, $\varphi\equiv 1$ on $6Q_k$, $\supp
\varphi\subset 10Q_k$ and $\alpha =\fint_{10Q_k}f$. Then
\begin{align*}
|Q_k\cap E(A\lambda)| &\leq |\{P\in Q_k: M_{5Q_k}[\tilN{(u-v)}]>\textstyle\frac{A\lambda}{2}\}|\\
&\quad + |\{P\in Q_k: M_{5Q_k}[\tilN{v}]>\textstyle\frac{A\lambda}{2}\}|\\
&\leq \frac{C}{(A\lambda)^{\bar{p}}}\int_{5Q_k}
\tilN{(u-v)}^{\bar{p}}\,d\sigma + \frac{C}{A\lambda} \int_{5Q_k}
\tilN{v}\,d\sigma = I+ II
\end{align*}
by the weak $(\bar{p},\bar{p})$ and the weak $(1,1)$ inequality.
We choose $\bar{p}>p$ so that $(D^*)_{\bar{p}'}$ still holds.
Since $(R)_{\HS}$ holds, Lemma \ref{MainHSEstimate} for $q=1$
implies for the second term
$$II\leq \frac{C}{A\lambda}||\varphi(f-\alpha)||_{\HS}\leq \frac{C}{A\lambda} |Q_k| M(M(|\nabla f|))(Q)$$
for any $Q\in 5Q_k$. Thus we can choose a $Q$ from $Q_k\cap \{M(M(|\nabla f|)\leq \gamma \lambda\}$ to get $II\leq \frac{C\gamma}{A}|Q_k|.$\\
For $I$ observe that $u-v-\alpha$ is a weak solution with
vanishing boundary data on $6Q_k$. For this term we use the Main
Lemma of \cite{She07}, namely the reverse H\"older inequality for
$\tilN{u}$.
\begin{lemma}\label{ReversePartII}[Theorem 2.9 in \cite{She07}]
Assume that $(D^*)_{p'}$ holds. Let $w$ be a weak solution which
vanishes on $\Delta_{4R}(Q)$. Then
$$\left(\fint_{\Delta_R(Q_0)}\tilN{w}^p\,d\sigma\right)^{\frac{1}{p}}\leq \fint_{\Delta_{4R}(Q_0)} N(\nabla w)\,d\sigma.$$
\end{lemma}
Hence it follows that
\begin{align*}
I & \leq \frac{C}{(A\lambda)^{\bar{p}}}|Q_k|\left(\fint_{6Q_k}\tilN{(u-v)}\,d\sigma\right)^{\bar{p}}\\
&\leq \frac{C}{(A\lambda)^{\bar{p}}}|Q_k| \left[\left(\fint_{6Q_k}\tilN{u}\,d\sigma\right)^{\bar{p}} + \left(\fint_{6Q_k} \tilN{v}\,d\sigma\right)^{\bar{p}}\right]\\
&\leq \frac{C}{(A\lambda)^{\bar{p}}}\left[\lambda^{\bar{p}} +
(\gamma\lambda)^{\bar{p}}\right]|Q_k|\leq
\frac{C}{A^{\bar{p}}}|Q_k|.
\end{align*}
To get the last line we have used the facts that $3Q_k\cap
E(\lambda)^c\ne 0$ as well as $Q_k\cap
\{M(M(|\nabla f|))\leq \gamma\lambda\} \neq
\emptyset$ and that $(R)_{\HS}$ holds. In the last step we hid
$\gamma$ into a generic constant $C$, we can do this since
$\gamma>0$ will be chosen small in the next step. Collecting all
estimates together we see that
\begin{align*}
|Q_k\cap E(A\lambda)| &\leq |Q_k|\left(\frac{C\gamma}{A} + \frac{C}{A^{\bar{p}}}\right)\\
&= |Q_k| (C\gamma \varepsilon^{\frac{1}{p}} + C\varepsilon^{\bar{p}/p})\\
&= |Q_k| \varepsilon^{1+\eta} (C\gamma \varepsilon^{\frac{1}{p} -
1-\eta} + C\varepsilon^{\eta}),
\end{align*}
for $\eta=\frac{1}{2}(\bar{p}/p -1)>0$. We now choose
$\varepsilon>0$ small enough to make the second term less than
$\frac{1}{2}$ and then choose $\gamma$ such that the first term is
smaller than $\frac{1}{2}$. Therefore
$$|Q_k\cap E(A\lambda)|\leq \varepsilon^{1+\eta}|Q_k|,$$
which finishes the proof.
\end{proof}
With (\ref{HSGoodLambda}) established the proof of the Main
Theorem in \cite{She07} implies the our main result. For
completeness we include the proof.
\begin{theorem}\label{RHSToRp} There exists $1<p<\infty$ such that $(R)_{\HS}$ implies $(R)_p$.
\end{theorem}
\begin{proof}
By Theorem \ref{RpToDp} there exists $1<p<\infty$ such that
$(D^*)_{p'}$ holds. We multiply (\ref{HSGoodLambda}) both sides
with $\lambda^{p-1}$ and integrate then over $(\lambda_0,\Lambda)$
to get
$$\int_{\lambda_0}^{\Lambda} |E(A\lambda)| \lambda^{p-1}\diff\lambda \leq \varepsilon^{1+\eta}\int_{\lambda_0}^{\Lambda} |E(\lambda)|\lambda^{p-1}\diff\lambda + C\int|\nabla f|^p\,d\sigma,$$
where for the last term we used the boundedness of the
Hardy-Littlewood maximal function on $L^p$ twice. Using the change
of variables $A\lambda\mapsto\lambda$ we get
$$\int_{A\lambda_0}^{A\Lambda} |E(\lambda)| \lambda^{p-1}A^{1-p}A^{-1}\diff\lambda \leq \varepsilon^{1+\eta}\int_{\lambda_0}^{\Lambda} |E(\lambda)|\lambda^{p-1}\diff\lambda + C\int|\nabla f|^p\,d\sigma.$$
By the definition of $A$ we have $A^{1-p}A^{-1}=\varepsilon$.
Therefore the previous inequality simplifies to
$$\int_{A\lambda_0}^{A\Lambda} |E(\lambda)| \lambda^{p-1}\diff\lambda \leq  \varepsilon^{\eta}\int_{\lambda_0}^{\Lambda} |E(\lambda)|\lambda^{p-1}\diff\lambda + C(\varepsilon)\int|\nabla f|^p\,d\sigma.$$
For $\varepsilon$ small enough such that $\varepsilon^{\eta}\leq
\frac{1}{2}$ and $\Lambda$ large enough such that $\Lambda \geq
A\lambda_0$, we can hide the part
$\varepsilon^{\eta}\int_{A\lambda_0}^{\Lambda}
|E(\lambda)|\lambda^{p-1}\diff\lambda$ on the left hand side to
get
$$\int_{A\lambda_0}^{A\Lambda} |E(\lambda)| \lambda^{p-1}\diff\lambda \leq  C\int_{\lambda_0}^{A\lambda_0} |E(\lambda)|\lambda^{p-1}\diff\lambda + C\int|\nabla f|^p\,d\sigma.$$
By adding $\int_0^{A\lambda_0}|E(\lambda)|
\lambda^{p-1}\diff\lambda$ on both sides we end up with
\begin{align}\label{LambdaRuns}
\int_{0}^{A\Lambda} |E(\lambda)| \lambda^{p-1}\diff\lambda \leq
C\int_{0}^{A\lambda_0} |E(\lambda)|\lambda^{p-1}\diff\lambda +
C\int|\nabla f|^p\,d\sigma.
\end{align}
By the definition of $\lambda_0$, the $(R)_{\HS}$-condition and
H\"older's inequality the first term of the right hand side is
bounded by
\begin{align*}
C\left(\int_{\partial\Omega}\tilN{u}\,d\sigma\right)^p \leq C
||f||_{\HS}^p\leq C ||f||_{H^{1,p}(\partial\Omega)}^p.
\end{align*}
Thus sending $\Lambda\rightarrow \infty$ in (\ref{LambdaRuns})
gives $\int_{\partial\Omega}(M(\tilN{u}))^p\leq
C||f||_{H^{1,p}(\partial\Omega)}^p$, i.e.
\begin{align}\label{YesYes}
||\tilN{u}||_{L^p(\partial\Omega)}\leq
C||f||_{H^{1,p}(\partial\Omega)}.
\end{align}
It remains to check that $||u||_{L^p(\Omega)}\leq
C||f||_{H^{1,p}(\partial\Omega)}$. By the usual splitting into the
positive end negative part, we can without loosing generality
assume that $f$ is non-negative. We have
$$||u||_{L^p(\Omega)} \leq C ||\tilN{u}||_{L^p(\partial\Omega)} + C ||u||_{L^1(\Omega_{R_0})}\leq C ||f||_{H^{1,p}(\partial\Omega)} + ||f||_{\HS}\leq C||f||_{H^{1,p}(\partial\Omega)}.$$
\end{proof}
The $p$ in Theorem \ref{RHSToRp} was determined by the $p'$ for
which $(D^*)_{p'}$ holds. Thus Theorem \ref{RHSToRp} allows to
conclude the following:
\begin{corollary}\label{EndpointIncluded}
Let $L$ be an elliptic operator with the elliptic measure of the
adjoint $L^*$ operator in $A_{\infty}(d\sigma)$. Then either
\begin{align*}
&\begin{cases}
& (i)_a \text{ } (D^*)_{p'}\text{ implies } (R)_p\text{ for all $p\in (1,\infty) $ for which $(D^*)_{p'}$ holds}\\
& (i)_b \text{ } (D^*)_{\text{BMO}}\text{ implies } (R)_{\HS}\\
\end{cases}\\
or&\\
&\begin{cases}
&(ii)_a \text{ } (R)_p \text{ is not solvable for any $p\in(1,\infty)$}\\
&(ii)_b \text{ } (R)_{\HS}\text{ is not solvable.}
\end{cases}
\end{align*}
\end{corollary}
It remains an open question whether the second alternative in
Corollary \ref{EndpointIncluded} does happen or whether
$(D^*)_{p'}$ always implies $(R)_p$. By Corollary
\ref{EndpointIncluded}, Theorem \ref{ZeroPart}, part of the proof
of Theorem \ref{RpToRHS} regarding the $||u||_{L^1(\Omega)}$ norm
and Lemma \ref{RpSmoothEnough}, we get the following:
\begin{corollary}\label{RpDpSimplified}
Assume that for all smooth Hardy-Sobolev $(1,\infty)$-atoms $f$
the weak solution $u$ of the equation $Lu=0$ with Dirichlet data
$f$ satisfies
\begin{align*}
\int_{8\Delta_R(Q)}\tilN{u}\leq C,
\end{align*}
where $\Delta_R(Q)$ is a surface ball on which the atom $f$ is
supported and $C$ is a constant independent of $f$. Then
$$(D^*)_{p'}\quad\text{ implies }\quad (R)_p.$$
\end{corollary}

\end{document}